\documentclass[11pt]{amsart}
\usepackage{amsmath, amsthm, amsfonts, hyperref, graphicx, ifpdf, mathrsfs}

\usepackage[dvipsnames,usenames]{color}
\usepackage{fancyhdr}
\newtheorem{df}{Definition}[section]
\newtheorem{thm}[df]{Theorem}
\newtheorem{question}[df]{Question}
\newtheorem*{Question*}{Question}

\newtheorem{prop}[df]{Proposition}
\newtheorem{rem}[df]{Remark}

\newtheorem{lem}[df]{Lemma}

\title{on the uniqueness of vortex equations and its geometric applications}
\author{Qiongling Li}

\address{
Centre for Quantum Geometry of Moduli Spaces (QGM)\\
Aarhus University\\
Ny Munkegade 118\\
8000 Aarhus C, Denmark}
\address{
Department of Mathematics\\
California Institute of Technology\\
1200 East California Boulevard\\
 Pasadena, US}
\email{qiongling.li@gmail.com}

\begin{document}

\begin{abstract}
We study the uniqueness of a vortex equation involving an entire function on the complex plane. As geometric applications, we show that there is a unique harmonic map $u:\mathbb{C}\rightarrow \mathbb{H}^2$ satisfying $\partial u\neq 0$ with prescribed polynomial Hopf differential; there is a unique affine spherical immersion $u:\mathbb{C}\rightarrow \mathbb{R}^3$ with prescribed polynomial Pick differential. We also show that the uniqueness fails for non-polynomial entire functions with finite zeros. 
\end{abstract}
\maketitle
\section{Introduction}
Let $\phi=\phi(z)$ be a nonzero entire function on $\mathbb{C}$ and $k\geq 2$ be an integer. We consider the following vortex equation on the complex plane
\begin{equation}\label{main}\triangle w=e^{w}-|\phi|^2 e^{-(k-1)w},\end{equation} where $\triangle=4\partial_z\partial_{\bar z}$.
In the case $k=2$, it is induced from the harmonic map equation from $\mathbb{C}$ to $\mathbb{H}^2$, which is extensively studied, see \cite{SchoenYau, WAN, Parabolic,HAN, HTTW}.  In the case $k=3$, it is Wang's equation for affine spherical immersions, see \cite{Wang, SimonWang,DumasWolf}. For general $k$, such equations are first considered in \cite{DumasWolf} and later generalized in \cite{DaiLi}. 

There is a geometric interpretation for Equation (\ref{main}). Let $\sigma=e^{w}|dz|^2$ be a conformal metric on $\mathbb{C}$. Equation (\ref{main})  becomes
\begin{equation*}2K_{\sigma}=-1+|\phi|_{\sigma}^2,
\end{equation*}
where $K_{\sigma}=-\frac{1}{2\sigma}\triangle\log \sigma$ is the curvature of $\sigma$. Roughly speaking, twice of the curvature of $\sigma$ differs from the curvature of the hyperbolic metric by the norm of $\phi dz^k$ with respect to the metric $\sigma$.

The first main theorem we show is 
\begin{thm}\label{mainmain1}
Let $\phi$ be a nonzero polynomial function on $\mathbb{C}$, there is a unique $C^{\infty}$ solution $w$ on $\mathbb{C}$ to Equation (\ref{main}). 
\end{thm}
As an application of the case $k=3$, the above theorem gives a negative answer to the question asked by Dumas and Wolf on page 1765 in \cite{DumasWolf}, which is one of the original motivations of our paper.
\begin{question}(Dumas-Wolf) \label{DWquestion}Does there exist an affine spherical immersion $f:\mathbb{C}\rightarrow \mathbb{R}^3$ with polynomial Pick differential whose Blaschke metric has positive curvature at some point?
\end{question}
For a nonzero entire function $\phi$, we first show the following proposition. 


\begin{prop}\label{main1}
Let $\phi$ be a nonzero entire function on $\mathbb{C}$, there is a unique $C^{\infty}$ solution $w$ on $\mathbb{C}$ to Equation (\ref{main}) such that $e^w|dz|^2$ is complete.
\end{prop}

So Theorem \ref{mainmain1} shows a stronger uniqueness result when $\phi$ is a polynomial. The proof of Proposition \ref{main1} is adapted on previous work of \cite{WAN,Parabolic,HAN,HTTW, DumasWolf} and uses techniques in \cite{BenoistHulin, Calabi, LiLiSimon} of the case $k=3$ where they show the completeness of the Blaschke metric implies that it is negatively curved.

We then study the case where $\phi$ is a non-polynomial entire function with finite zeros.  Any such entire function is of the form $\phi(z)=P(z)e^{Q(z)}$, where $P(z)$ is a polynomial and $Q(z)$ is any nonconstant entire function. In this case, $\phi$ has an essential singularity at infinity. For such $\phi$, the uniqueness of the solution fails. The second main theorem is
\begin{thm}\label{mainmainmain1}Suppose $\phi$ is a non-polynomial entire function with finite zeros, then Equation (\ref{main}) has at least two $C^{\infty}$ solutions $w_1>w_2$ where $e^{w_1}|dz|^2$ is a complete metric on $\mathbb{C}$ while $e^{w_2}|dz|^2$ is an incomplete metric on $\mathbb{C}$. 
\end{thm}

\begin{rem}
It would be interesting to know what happens when $\phi$ has infinite zeros, for example, $\sin \pi z$. 
\end{rem}

\subsection{Geometric Applications}
\subsubsection{Harmonic maps with prescribed Hopf differential}
Consider a harmonic map $u:\mathbb{C}\rightarrow \mathbb{H}^2$. Suppose $\partial u$ does not vanish. Writing $w=\frac{1}{2}\log |\partial u|^2$ and let $q$ be the Hopf differential of $u$. Then $u$ being harmonic is equivalent to $q$ being holomorphic and that the following equation holds on $\mathbb{C}$:
\[\triangle w=e^{2w}-|q|^2e^{-2w}.\]

As an application of Theorem \ref{mainmain1} and \ref{mainmainmain1}, we have
\begin{thm} Suppose $q$ is a nonzero entire function on $\mathbb{C}$ with finite zeros, then\\
(1) if $q$ is a polynomial, there exists a unique harmonic map $u:\mathbb{C}\rightarrow \mathbb{H}^2$ satisfying $\partial u\neq 0$ with its Hopf differential $qdz^2$ up to an isometry of $\mathbb{H}^2$. Moreover, it is an orientation-preserving harmonic embedding.\\
(2) if $q$ is not a polynomial, there are at least two harmonic maps $u_1,u_2:\mathbb{C}\rightarrow \mathbb{H}^2$ satisfying $\partial u_i\neq 0$ not related by an isometry of $\mathbb{H}^2$ with the given $qdz^2$ as their Hopf differential.
\end{thm}

\subsubsection{Affine spherical immersions with prescribed Pick differential}
Consider a locallystrictly convex immersed hypersurface $f:M\rightarrow \mathbb{R}^3$. The affine differential geometry associates to such a locally convex immersed surface a special transverse vector field $\xi$, the affine normal. Being a hyperbolic affine spherical immersion means the affine normal of each image point $f(p)$ meets at a point (the center), which lies in a convex side of the hypersurface. 

Relative to the affine normal, there are two important objects on $M$: (1) the second fundamental form induces a Riemannian metric $h$, called the Blaschke metric; (2) a holomorphic cubic differential $Udz^3$, called the Pick differential, measuring the difference between the induced connection of the immersion and the Levi-Civita connection with respect to $h$. 
 
In the case when $(M,h)$ is conformal to $(\mathbb{C},|dz|^2)$, we reparametrize the affine spherical immersion as $f:\mathbb{C}\rightarrow \mathbb{R}^3$. Following Wang \cite{Wang} and Simon-Wang \cite{SimonWang}, $f$ being the affine spherical immersion is equivalent to $U$ being holomorphic and the Blaschke metric $e^w|dz|$ satisfies
\[\triangle w=2e^w-4|U|^2e^{-2w}.\]

As an application of Theorem \ref{mainmain1} and \ref{mainmainmain1}, we have
\begin{thm} Suppose $U$ is a nonzero entire function on $\mathbb{C}$ with finite zeros, then\\
(1)  if $U$ is a polynomial, there is a unique affine spherical immersion $f:\mathbb{C}\rightarrow \mathbb{R}^3$ with its Pick differential as $Udz^3$ up to a projective transformation of $\mathbb{R}^3$. Moreover, it is complete, properly embedded and its Blaschke metric is negatively curved.\\
(2) if $U$ is not a polynomial, there are at least two affine spherical immersions $f_1,f_2:\mathbb{C}\rightarrow \mathbb{R}^3$ not related by a projective transformation of $\mathbb{R}^3$ with the given $Udz^3$ as their Pick differential.
\end{thm}

\subsection*{Acknowledgments} The author wishes to thank Vlad Markovic, Song Dai and Mike Wolf for helpful discussions. The author is supported by the center of excellence grant `Center for Quantum Geometry of Moduli Spaces' from the Danish National Research Foundation (DNRF95). She also acknowledges the support from U.S. National Science Foundation grants DMS 1107452, 1107263, 1107367 ``RNMS: GEometric structures And Representation varieties'' (the GEAR Network).

\section{Main tools}
We collect here the main tools for complete manifolds used in the paper.
\subsection*{Omori-Yau Maximum Principle}(\cite{Omori, Yau75})
Suppose $(M,h)$ is a complete manifold with Ricci curvature bounded from below. Then for a $C^2$ function $u:M\rightarrow \mathbb{R}$ bounded from above, there exists a $m_0\in \mathbb{N}$ and a family of points $\{x_m\}\in M$ such that for each $m\geq m_0$,
\[u(x_m)\geq \sup u-\frac{1}{m}, \quad |\nabla_hu(x_m)|\leq \frac{1}{m}, \quad \triangle_h u(x_m)\leq \frac{1}{m},\]
where $\nabla_h$, $\triangle_h$ are the gradient and the Laplacian with respect to the background metric $h$ respectively.

\subsection*{Cheng-Yau Maximum Principle} (\cite{ChengYau})
Suppose $(M,h)$ is a complete manifold with Ricci curvature bounded from below. Let $u$ be a $C^2$-function defined on $M$ such that 
$\triangle_h u\geq f(u),$ where $f:\mathbb{R}\rightarrow \mathbb{R}$ is a function. Suppose there is a continuous positive function $g(t):[a,\infty)\rightarrow \mathbb{R}_+$ such that \\
(i) $g$ is non-decreasing; \\(ii) $\lim\inf_{t\rightarrow \infty}\frac{f(t)}{g(t)}>0$;\\ 
(iii) $\int_a^{\infty}(\int_b^tg(\tau)d\tau)^{-\frac{1}{2}}dt<\infty$, for some $b\geq a$,\\
 then the function $u$ is bounded from above. Moreover, if $f$ is lower semi-continuous, $f(\sup u)\leq 0$.

\subsection*{Supersolution and subsolution method}
Let $F(x,u)$ be a $C^{\infty}$ function defined on $M\times \mathbb{R}$ where $(M,h)$ is a complete Riemannian manifold. Furthermore, assume that 
\[\frac{\partial F}{\partial u}>0 \quad\quad\quad \text{for all} \quad (x,u)\in M\times \mathbb{R}.\] 

\begin{lem}\label{supersub} (see \cite{WAN})
If there exists $w_{+},w_-\in C^0(M)\cap H_{loc}^1(M)$ such that $w_-\leq w_+$ and 
\[\triangle w_+\leq F(x,w_+),\quad\quad \triangle w_-\geq F(x,w_-)\] are satisfied weakly (in the sense of distribution), then there is a $C^{\infty}$ solution $w$ to the equation \[\triangle w= F(x,w),\]
such that $w_-\leq w\leq w_+.$
\end{lem}
In our setting $F(x,u)=e^u-|\phi|^2e^{-(k-1)u}$, satisfying $\frac{\partial F}{\partial u}>0$.

\section{Complete solutions}
In this section, given an entire function $\phi$ on $\mathbb{C}$, we show the existence and uniqueness of a solution $w$ to
\begin{equation*}\triangle w=e^{w}-|\phi|^2 e^{-(k-1)w}\end{equation*}
 such that the metric $e^w|dz|^2$ is complete. 

\begin{thm}\label{CompleteSolution}
Let $\phi$ be a nonzero entire function on $\mathbb{C}$, there is a unique $C^{\infty}$ solution $w$ to Equation (\ref{main})
on $\mathbb{C}$ such that the metric $e^w|dz|^2$ is complete. 
\end{thm}

\begin{rem} Under the substitution by $w':=cw+\log d$, the above theorem can also be applied to the equation 
\begin{equation*}c\cdot\triangle w=de^{cw}-\frac{1}{d^{k-1}}|\phi|^2 e^{-c(k-1)w},\end{equation*} for any $c,d>0$. So through the paper, whenever referring to the above theorem, we use it with different constants without further explanation.
\end{rem}
We prove the above theorem in the end of this section.

To show the existence of solutions on $\mathbb{C}$, we start with the existence of solutions on a domain $(U,h)$, which lifts to the unit disk $\mathbb{D}$ with the Poincar\'e metric. The proof of the following proposition is easily adapted from the proof of Proposition 4  in \cite{Parabolic}.
\begin{prop}\label{hyperbolicdomain}
Let $\phi dz^k$ be a holomorphic $k$-differential on $(U,h)$, then Equation (\ref{main}) has a unique $C^{\infty}$ solution $w$ on $U$ such that $e^w|dz|^2$ is a complete metric on $U$ and $|\phi|^2 e^{-kw}\leq 1$. Moreover, $w$ satisfies $h\leq e^w|dz|^2$.
\end{prop}

The following existence theorem mainly follows from \cite{Parabolic}, where they work with the case $k=2$. We include the proof here. 

\begin{prop}(Existence)
Let $\phi$ be a nonzero entire function on $\mathbb{C}$, there is a $C^{\infty}$ solution $w$ to Equation (\ref{main})
on $\mathbb{C}$ such that $e^w|dz|^2$ is complete. 
\end{prop}

\begin{proof}
Without loss of generality, we assume $\phi(0)\neq 0$. Choose $R_1$ such that the zeros of $\phi$ are outside the ball $B_{R_1}$. Choose $R_2<R_1$. Let $u_1$ be the unique solution on $B_{R_1}$ and $u_2$ be the unique solution in the complement of the closed disk $\overline{B}_{R_2}$ as described in Proposition \ref{hyperbolicdomain}.

We first construct a supersolution $w_{+}$ as follows:\\
on $\{|z|<R_2\}$, let $w_{+}=u_1$;\\
on $\{R_2\leq |z|\leq R_1\}$, let $w_{+}=\min\{u_1,u_2\}$;\\
on $\{|z|>R_1\}$, let $w_+=u_2$.\\
By Proposition \ref{hyperbolicdomain}, since $e^{u_1}|dz|^2$ dominates the hyperbolic metric on $B_{R_1}$, the metric $e^{u_1}|dz|^2$ blows up at boundary of $B_{R_1}$. Then at the neighbourhood of $\partial B_{R_1}$, $w_{+}=u_2$. Similarly, since $e^{u_2}|dz|^2$ dominates the hyperbolic metric on the complement of $B_{R_2}$, the metric $e^{u_2}|dz|^2$ blows up at boundary of $B_{R_2}$. Then at the neighbourhood of $\partial B_{R_2}$, $w_{+}=u_1$. Hence $w_+$ is a continuous function. Since $u_1,u_2$ are both supersolutions on the annulus, then the minimum of $u_1$ and $u_2$ is a supersolution in the weak sense. Therefore $w_+\in C^0(M)\cap H^1_{loc}(M)$ is a supersolution.

Secondly, we construct a subsolution $w_-$. Choose $R_1', R_2'$ such that $R_2<R_2'<R_1'<R_1$. On the annulus $\{R_2'\leq |z|\leq R_1'\}$, $u_2$ has a maximum and $\frac{2}{k}\log |\phi|$ has a minimum using that the zeros of $\phi$ are all outside $B_{R_1}$. Then there exists a constant $c>0$ such that 
\[\frac{2}{k}\log|\phi|\geq u_2(z)-c\] holds on the annulus  $\{R_2'\leq |z|\leq R_1'\}$.

We construct the subsolution $w_{-}$ as follows:\\
onn $|z|<R_1'$, let $w_{-}=\frac{2}{k}\log|\phi|$;\\
onn $|z|\geq R_2'$, let $w_{-}=\max\{\frac{2}{k}\log |\phi|, u_2(z)-c\}.$
Our choice of $c$ assures $w_{-}$ is continuous. Since $\frac{2}{k}\log |\phi|,u_2-c$ are both subsolutions on $\{|z|\geq R_2'\}$, then the maximum of $\frac{2}{k}\log |\phi|$ and $u_2-c$ is a subsolution in the weak sense. Therefore $w_-\in  C^0(M)\cap H^1_{loc}(M)$ is a subsolution.

By Proposition \ref{hyperbolicdomain} $e^{u_i}\geq|\phi|^{\frac{2}{k}}$, for $i=1,2$. It is clear that $w_-\leq w_+$. Applying Lemma \ref{supersub}, there is a $C^{\infty}$ solution $w$ satisfying $w_-\leq w\leq w_+$. Given a path $\alpha$ diverging to $\infty$ on $\mathbb{C}$, then 
\[\int_{\alpha}e^{\frac{1}{2}w}|dz|\geq \int_{\alpha}e^{\frac{1}{2}w_-}|dz|=e^{-\frac{1}{2}c}\int_{\alpha\cap B_{R_1'}^c}e^{\frac{1}{2}u_2}|dz|=\infty,\]by the completeness of $e^{u_2}|dz|^2$. Hence the metric $e^w|dz|^2$ is complete.
\end{proof}

Now we show a weaker uniqueness result under two conditions that the metric $e^w|dz|^2$ is complete and $|\phi|^2e^{-kw}\leq 1$. The method is similar to the proof in \cite{WAN}, \cite{DumasWolf}.
\begin{lem} \label{weakunique}
Let $\phi$ be an entire function on $\mathbb{C}$. If there is a $C^{\infty}$ solution $w$ to Equation (\ref{main})
on $\mathbb{C}$ such that $e^w|dz|^2$ is complete and $|\phi|^2e^{-kw}\leq 1.$ Then for any another $C^{\infty}$ solution $w_1$, either $w_1<w$ on $\mathbb{C}$ or $w_1\equiv w$ on $\mathbb{C}$.
\end{lem}
\begin{proof} We first have
\begin{equation*}
\triangle(w_1-w)=(e^{w_1}-e^{w})-|\phi|^2(e^{-(k-1)w_1}-e^{-(k-1)w}).
\end{equation*}
Denote $\eta=w_1-w$, using the background metric $g=e^w|dz|^2$, we rewrite the above equation as
\begin{equation*}
\triangle_{g}\eta=(e^{\eta}-1)-|\phi|^2e^{-kw}(e^{-(k-1)\eta}-1)\geq (e^{\eta}-1)-e^{-(k-1)\eta},
\end{equation*}
where $\triangle_{g}=\frac{1}{e^w}\triangle$ and the inequality uses $|\phi|^2e^{-kw}\leq 1$.

 Since the metric $g=e^w|dz|^2$ is complete and the curvature $K_g=-\frac{1}{2}\triangle_gw\geq -\frac{1}{2}+\frac{1}{2}|\phi|^2e^{-kw}\geq -\frac{1}{2}$, we can apply the Cheng-Yau maximum principle. Choosing $f(t)=e^t-1-e^{-(k-1)t}$ and $g(t)=e^t$, one can check (i) $g$ is increasing; (ii) $\lim\inf_{t\rightarrow \infty}\frac{f(t)}{g(t)}=1$; (iii) $\int_0^{\infty}(\int_0^tg(\tau)d\tau)^{-\frac{1}{2}}dt<\infty$. Hence we obtain that $\eta$ is bounded from above. 

By the Omori-Yau maximum principle, since $\eta$ is bounded from above, there exists a sequence of points $\{x_m\}$ such that there exists $m_0$ such that for $m\geq m_0$,
$\triangle_{g}\eta (x_m)<\frac{1}{m}$, $\eta(x_m)\geq \sup \eta-\frac{1}{m}$.

Then at point $x_m$,
\begin{equation*}
\frac{1}{m}\geq (e^{\sup \eta-\frac{1}{m}}-1)-|\phi|^2e^{-kw}(x_m)(e^{-(k-1)(\sup \eta-\frac{1}{m})}-1)
\end{equation*}
Suppose $\sup\eta> 0$ and then $\sup \eta>\frac{1}{M}$ for some $M>0$. Then for $m>\max\{m_0,M\}$, 
the term $-|\phi|^2e^{-kw}(x_m)(e^{-(k-1)(\sup \eta-\frac{1}{m})}-1)$ is positive. Therefore
\begin{equation*}
\frac{1}{m}\geq e^{\sup \eta-\frac{1}{m}}-1.
\end{equation*}
As $m\rightarrow \infty$,
\begin{equation*}
0\geq e^{\sup \eta}-1.
\end{equation*}
Therefore, $\sup\eta\leq 0$, contradiction. 

So $\sup\eta\leq 0$ and then $w_1\leq w$. 

Consider the equation for $\eta$
\begin{equation*}
\triangle_{g}\eta=(e^{\eta}-1)-|\phi|^2e^{-kw}(e^{-(k-1)\eta}-1)
\end{equation*}

Applying the strong maximum principle to the equation for $\eta$. Since $\eta\leq 0$, then $\eta<0$ on $\mathbb{C}$ or $\eta\equiv 0$. The lemma follows.
\end{proof}

We are ready to show Theorem \ref{CompleteSolution}.

\begin{proof}(of Theorem \ref{CompleteSolution})
It suffices to show the uniqueness. Using Lemma \ref{weakunique}, it is enough to show that $|\phi|^2e^{-kw}\leq 1$.
Consider the function $h=\frac{|\phi|^2}{e^{kw}}$ and we calculate the equation for $\tau=\log (h+1)$.
Calculate $h_z=(\phi_z-kw_z\phi)\bar\phi e^{-kw}, h_{\bar z}=(\bar\phi_{\bar z}-k\bar w_{\bar z}\phi)\phi e^{-kw}$, and 
$h_{z\bar z}=-kw_{z\bar z}h+|\phi_z-kw_z\phi|^2e^{-kw}$.

Then \begin{eqnarray*}
\triangle\tau&=&4\frac{h_{z\bar z}(h+1)-h_zh_{\bar z}}{(h+1)^2}\\
&=&-4\frac{khw_{z\bar z}}{h+1}+4\frac{|\phi_z-kw_z\phi|^2e^{-kw}}{(h+1)^2}\\
&\geq&-4\frac{khw_{z\bar z}}{h+1}=-kh(e^w-|\phi|^2e^{-(k-1)w})e^{-\tau}\\
&=&-kh(1-|\phi|^2e^{-kw})e^{-\tau}e^w\\
&=&-k(e^{\tau}-1)(2-e^{\tau})e^{-\tau}e^w\\
&=& k(e^{\tau}-3+2e^{-\tau})e^w.
\end{eqnarray*}
Let $g=e^w|dz|^2$ be the new background metric on $\mathbb{C}$, the above equation is
\[\triangle_g\tau\geq k(e^{\tau}-3+2e^{-\tau}).\] Since the metric $e^w|dz|^2$ is complete and the curvature $K_g\geq -\frac{1}{2}$, we can apply the Cheng-Yau maximum principle. By choosing $f(t)=k(e^t-3+2e^{-t})$ and $g(t)=e^t$, one can check 
\\
(i) $g$ is increasing; (ii) $\lim\inf_{t\rightarrow \infty}\frac{f(t)}{g(t)}=1$; (iii) $\int_0^{\infty}(\int_0^tg(\tau)d\tau)^{-\frac{1}{2}}dt<\infty$.
\\
We then have that $\tau$ is bounded from above and moreover $f(\sup \tau)\leq 0$. Hence $e^{\sup \tau}\leq 2$ and then $h=\frac{|\phi|^2}{e^{kw}}\leq 1$.
\end{proof}

\begin{prop}\label{strictlynegative}
If $\phi$ is not a constant, then the unique solution $w$ to Equation (\ref{main}) such that $e^w|dz|^2$ is complete satisfies $|\phi|^2e^{-kw}<1$.
\end{prop}
\begin{proof}
Consider the function $\sigma=\log (\frac{|\phi|^2}{e^{kw}})$ which is well-defined away from the zeros of $\phi$. We have the equation for $\sigma$
  \[\triangle_g\sigma=k(|\phi|^2e^{-kw}-1)=k(e^{\sigma}-1).\]
Since the right hand side is monotone increasing with respect to $\sigma$, one can apply the strong maximum principle to this equation.
Since $\sigma\leq 0$, then either $\sigma<0$ on $\mathbb{C}$ or $\sigma\equiv 0$. 
\end{proof}

\section{Discussion on the uniqueness}
In this section, we consider the uniqueness separately for the case $\phi$ is a polynomial and the case $\phi$ is not a polynomial. 

Firslty, we restrict ourselves to work with polynomials and show the uniqueness of solutions to Equation (\ref{main}) without any requirement. 
\begin{thm}\label{Unique}
If $\phi$ be a polynomial function on $\mathbb{C}$, then there is a unique solution to Equation (\ref{main}) on $\mathbb{C}$.
\end{thm}

\begin{proof} The existence is shown in the previous section. It suffices to show the uniqueness. We first show that $w$ has a lower bound. 
Since $\phi$ is a polynomial, there exists a large enough $R$ such that outside the ball $B_R$, $|\phi|^2\geq 1$. Outside $B_R$, we have the inequality 
\begin{equation*}
\triangle w\leq e^w-e^{-(k-1)w}.
\end{equation*}

Denote $B=\max_{B_R} (-w)$ and define the function $f(t)$ on $\mathbb{R}$ as: for $t\leq B, f=-e^{-t}$; for $t>B, f=e^{(k-1)t}-e^{-t}$. Therefore we have 
\begin{equation*}
\triangle(-w)\geq f(-w).
\end{equation*} Choose $g(t)=e^{(k-1)t}$, one can check 
\\
(i) $g$ is increasing; (ii) $\lim\inf_{t\rightarrow \infty}\frac{f(t)}{g(t)}=1$; (iii) $\int_0^{\infty}(\int_0^tg(\tau)d\tau)^{-\frac{1}{2}}dt<\infty$.
\\
Since the background metric is complete and has curvature $0$, we can apply the Cheng-Yau maximum principle. Hence we have $\sup -w<\infty$ and $w$ is bounded from below. Then the metric $e^{w}|dz|^2$ is complete. By Theorem \ref{CompleteSolution}, the solution is unique.
\end{proof}

Next, we study the case when $\phi$ is a non-polynomial entire function with finite zeros. In this case we show that Equation (\ref{main}) has at least two distinct solutions. 
\begin{thm}\label{NonUnique}
Let $\phi$ be an entire function on $\mathbb{C}$ with finite zeros. If $\phi$ is not a polynomial, then Equation (\ref{main}) has at least two solutions $w_1>w_2$ where $e^{w_1}|dz|^2$ defines a complete metric and $e^{w_2}|dz|^2$ defines an incomplete metric.  Moreover, $|\phi|^2e^{-kw_1}< 1$ and $|\phi|^2e^{-kw_2}\leq 1$.
\end{thm}
 
The main idea is that we construct a solution using another pair of a supersolution and a subsolution in the following Lemma \ref{Another} if $\phi$ is with finite zeros. When the differential $\phi$ is not a polynomial, the solution constructed here differs from the one constructed in Theorem \ref{CompleteSolution}.  
\begin{lem}\label{Another}
Let $\phi$ be a nonzero entire function on $\mathbb{C}$ with finite zeros, there is a $C^{\infty}$ solution $w$ to Equation (\ref{main})
on $\mathbb{C}$ such that there are constants $R>0, a>1$ such that for $|z|\geq R, $\[|\phi|^{\frac{2}{k}}\leq e^w\leq a|\phi|^{\frac{2}{k}}.\]
\end{lem}
\begin{proof}(of Lemma \ref{Another})
If $\phi$ has no zero, then $w=\frac{2}{k}\log |\phi|$ is a solution.    \\
If $\phi$ has finite zeros, we can construct a supersolution and a subsolution. 
Since $\phi$ has finite zeros, choose $R_2$ such that the ball $B_{R_2}=\{|z|<R_2\}$ contains all zeros of $\phi$. Choose $R_2<R_2'<R_1'<R_1$. Let $u_1$ be the unique solution on the domain $B_{R_1}$ defined in Proposition \ref{hyperbolicdomain}.

On the annulus $\{R_2'\leq |z|\leq R_1'\}$, $u_1$ has a maximum and $\frac{2}{k}\log |\phi|$ has a minimum using the fact that the zeros of $\phi$ are all inside $B_{R_2}$. Then there exists a constant $c>0$ such that 
\[\frac{2}{k}\log|\phi|\geq u_1(z)-c\] holds on the annulus  $\{R_2'\leq |z|\leq R_1'\}$.

We first construct the supersolution $w_{+}$ as follows:\\
on $\{|z|<R_2'\}$, let $w_+=u_1$;\\
on the annulus $\{R_2'\leq |z|\leq R_1\}$, let $w_{+}=\min\{u_1, \frac{2}{k} \log |\phi|+c\}$;\\
outside $B_{R_1}$, let $w_{+}=\frac{2}{k} \log |\phi|+c$.\\
On the annulus $\{R_2'\leq|z|\leq R_1'\}$, $u_1\leq \frac{2}{k}\log|\phi|+c$. Therefore, in the neighbourhood of $\partial B_{R_2'}$, $w_+=u_1$ is continuous. 
Since the metric $e^{u_1}|dz|^2$ dominates the hyperbolic metric on $B_{R_1}$, $u_1$ blows up on $\partial B_{R_1}$. Then in the neighbourhood of $\partial B_{R_1}$, $w_+=\frac{2}{k} \log |\phi|+c$ is continuous. 
Both $u_1$ and $\frac{2}{k}\log |\phi|+c$ are supersolutions, so their minimum is still a supersolution in the weak sense. Therefore $w_+\in C^0(\mathbb{C})\cap H^1_{loc}(\mathbb{C})$ is again a supersolution. 

Secondly, we construct the subsolution $w_{-}$ as follows:\\
on $\{|z|<R_1'\}$, let $w_{-}=\max\{u_1-c, \frac{2}{k} \log |\phi|\}$;\\
outside $B_{R_1'}$, let $w_{-}=\frac{2}{k} \log |\phi|$.\\
On the annulus $\{R_2'\leq|z|\leq R_1'\}$, $u_1-c\leq \frac{2}{k}\log|\phi|$. Then in the neighbourhood of $\partial B_{R_1'}$, $w_{-}=\frac{2}{k} \log |\phi|$ is continuous. Both $u_1-c$ and $\frac{2}{k}\log |\phi|$ are subsolutions, so their maximum is still a subsolution in the weak sense. Therefore $w_-\in C^0(\mathbb{C})\cap H^1_{loc}(\mathbb{C})$ is again a subsolution. 

One can check that $w_-\leq w_+$ as follows:\\
on $B_{R_2'}$, $w_-=\max\{u_1-c, \frac{2}{k} \log |\phi|\}\leq u_1=w_+$, using $\frac{|\phi|^{\frac{2}{k}}}{e^{u_1}}<1$;\\
on the annulus $\{R_2'\leq |z|\leq R_1'\}$, \\$w_{-}=\max\{u_1-c, \frac{2}{k} \log |\phi|\}\leq w_{+}=\min\{u_1, \frac{2}{k} \log |\phi|+c\}$;\\
on the annulus $\{R_1'\leq |z|\leq R_1\}$, \\$w_{-}= \frac{2}{k} \log |\phi|\leq w_{+}=\min\{u_1, \frac{2}{k} \log |\phi|+c\}$, using $\frac{|\phi|^{\frac{2}{k}}}{e^{u_1}}<1$;\\
outside $B_{R_1}$, let $w_{-}=\frac{2}{k} \log |\phi|\leq w_{+}=\frac{2}{k} \log |\phi|+c$.

Applying Lemma \ref{supersub}, there is a $C^{\infty}$ solution $w$ satisfying $w_{-}\leq w\leq w_+$. Note when $|z|>R_1'$, we have $\frac{2}{k} \log |\phi|=w_-\leq w\leq w_+=\frac{2}{k} \log |\phi|+c.$
\end{proof}

The following lemma follows from Lemma 9.6 in \cite{Osserman}. For readers' convenience, we include the argument here.
\begin{lem}\label{PolynomialComplete} Let $\phi$ be a nonzero entire function on $\mathbb{C}$ with finite zeros, the solution $w$ in Lemma \ref{Another} defines a complete metric $e^w|dz|^2$ if and only if $\phi$ is a polynomial.
\end{lem}
\begin{proof}
If $\phi$ is a polynomial, for a large enough $R$, $\phi(z)>1$ for all $|z|>R$. Then the metric $e^w|dz|^2> |\phi|^{\frac{2}{k}}|dz|^2\geq |dz|^2$ for $|z|>\max\{R,R_1\}.$ Hence the metric $e^w|dz|^2$ is complete.

 Suppose the metric $e^w|dz|^2$ is complete. Then for every path $C$ which diverges to infinity, we have 
\begin{equation}\label{pathcomplete}\int_C|\phi(z)|^{\frac{1}{k}}|dz|=\infty.\end{equation}
Since $\phi(z)$ has finite zeros, it is of the form $P(z)e^{G(z)}$, where $P(z)$ is a polynomial and $G(z)$ is an entire function. Then there exists a positive integer $N$ such that $|\phi(z)|^{\frac{1}{k}}\leq O(|z^N e^{G(z)}|)$. Define the entire function $F(z)$ as\[F(z)=\int z^Ne^{G(z)}dz,\quad F(z)=0.\]
So $F(z)$ vanishes at $0$ of order $N+1$. There exists a single-value branch $\zeta(z)=[F(z)]^{\frac{1}{N+1}}$ in a neighborhood of $z=0$ satisfying $\zeta'(z)\neq 0$. Hence we have an inverse $z(\zeta)$ in a neighborhood of $\zeta=0$. 

We claim the inverse map can extend to the whole $\zeta$-plane. Suppose it is not the case, there would be a point $\zeta_0$ satisfying $|\zeta_0|\leq R$ over which $z(\zeta)$ cannot be extended. Consider the path $C$ as image of the segment from $0$ to $\zeta_0$ under the map $z(\zeta)$, 
\[\int_C|z^Ne^{G(z)}| |dz|=\int_C|F'(z)| |dz|=R^{N+1}.\] By Equation (\ref{pathcomplete}), the path $C$ does not diverge to infinity. Then along this path there exists a sequence $z_n\rightarrow z_0$ such that $\zeta(z_n)\rightarrow\zeta_0$. But since $F'(z_0)$ is nonzero, we can extend $z(\zeta)$ over $\zeta_0$. We thus have the inverse map $z(\zeta)$ is defined on the whole $\zeta$-plane.

Thus $\zeta(z)$ is entire and invertible, hence $\zeta(z)=Az$ for $z\neq 0$. So $F(z)=(z/A)^{N+1}$. And then $G(z)$ is a constant. Therefore $\phi$ is a polynomial.
\end{proof}
\begin{proof}(of Theorem \ref{NonUnique})
By Lemma \ref{PolynomialComplete}, the solution $w$ constructed in Lemma \ref{Another} defines a complete metric if and only if $\phi$ is polynomial. 
Then, if $\phi$ is not a polynomial, the solution constructed in Lemma \ref{Another} defines an incomplete metric $e^w|dz|^2$. By Theorem \ref{CompleteSolution}, there is also another solution $w_1$ such that $e^{w_1}|dz|^2$ is complete. Hence $w\neq w_1$. By Lemma \ref{weakunique}, $w<w_1$. 

It is clear that $|\phi|^2e^{-kw}\leq 1$ and by Proposition \ref{strictlynegative}, $|\phi|^2e^{-kw_1}<1$. \end{proof}

We also prove the following property of a solution to Equation (\ref{main}).
\begin{prop} For any nonzero entire function $\phi$ on $\mathbb{C}$, there is no solution $w$ to Equation (\ref{main}) satisfies $|\phi|^2e^{-kw}\leq \delta,$ for some positive constant $\delta<1$.
\end{prop}
\begin{proof}
Suppose there exists a solution $w$ such that $|\phi|^2e^{-kw}\leq \delta<1$, then the solution $w$ satisfying
\[\triangle w=e^w(1-|\phi|^2e^{-kw})\geq (1-\delta)e^w.\]
Applying the Cheng-Yau maximum principle, we obtain that $w$ is bounded from above. Therefore $|\phi|^2\leq \delta e^{kw}<\infty$. Hence $\phi$ is bounded. By Liouville theorem, $\phi$ is a constant function $c\neq 0$. But in this case we can show that $|c|^2e^{-kw}\equiv1$.
 Consider the equation for $-w$, \[\triangle (-w)=|c|^2e^{-(k-1)w}-e^w.\] Applying the Cheng-Yau maximum principle by choosing $f(t)=|c|^2e^{(k-1)t}-e^{-t}$ and $g(t)=e^{(k-1)t}$, we obtain that $-w$ is bounded from above. Therefore $e^w|dz|^2$ is complete. We also see the constant function $\frac{2}{k}{\log|c|^2}$ is a solution. Using the uniqueness of complete solutions in Lemma \ref{CompleteSolution}, we obtain that $|c|^2e^{-kw}$ is identically $1$, contradiction.
\end{proof}

\section{Harmonic maps}
We consider harmonic maps $u:\mathbb{C}\rightarrow (\mathbb{H}^2,h)$. Call $|\partial u|^2$ the $\partial$-energy density. Following Schoen-Yau \cite{SchoenYau}, when $\partial u\neq 0$, the following equation holds
\begin{equation}\label{harmonicprevious}\triangle \log |\partial u|^2=2(|\partial u|^2-|\overline{\partial}u|^2).\end{equation}
Writing $w=\frac{1}{2}\log |\partial u|^2$. Let $q$ be the quadratic differential defined by $q(z)=(u^*h)^{2,0}$, called the Hopf differential. The harmonicity of $u$ is equivalent to the Hopf differential $q$ being holomorphic and that Equation (\ref{harmonicprevious}) relates $(w,q)$ as follows
\begin{equation}\label{harmonic}\triangle w=e^{2w}-|q|^2e^{-2w}.\end{equation} The Jacobian $J$ of the harmonic map is $J=|\partial u|^2-|\overline{\partial}u|^2=e^{2w}-|\phi|^2e^{-2w}$. 

\begin{thm}\label{UniqueHarmonic} Given a nonconstant polynomial function $q$ on $\mathbb{C}$, there exists a unique harmonic map $u:\mathbb{C}\rightarrow \mathbb{H}^2$ satisfying $\partial u\neq 0$ with its Hopf differential $qdz^2$ up to an isometry of $\mathbb{H}^2$. Moreover, it is an orientation-preserving harmonic embedding.
\end{thm}

\begin{proof}
The existence is standard as shown in \cite{Parabolic}. Given such a $q$, by Theorem \ref{Unique}, there exists a solution $w$ to Equation (\ref{harmonic}) such that $e^{2w}|dz|^2$ is complete. We then construct a harmonic map giving back such data through constructing space-like constant mean curvature immersions in the Minkowski 3-space. The Minkowski $3$-space $\mathbb{M}^{2,1}$ is  $\mathbb{R}^2\times \mathbb{R}^1$ endowed with the metric $ds^2=(dx^1)^2+(dx^2)^2-(dx^3)^2$. Applying Milnor's method \cite{Milnor}, we develop a constant mean curvature immersion from $\mathbb{C}$ to $\mathbb{M}^{2,1}$ using the pair $(w,q)$ as follows: define $h_{ij}$ such that 
\[h_{11}+h_{22}=e^{2w},\quad h_{11}-h_{22}=2\Re q, \quad h_{12}=h_{21}=-\Im q;\]
 and define the first and second fundamental form $I$ and $II$ as \[I=e^{2w}|dz|^2,\quad II=h_{ij}dx^idx^j.\] Then the form $I$ and $II$ satisfy the integrability condition of the Fundamental Theorem of Differential Geometry using the equation of $q$ and $w$. Since the domain $\mathbb{C}$ is simply connected, we can develop a constant mean curvature space-like immersion from $\mathbb{C}$ to $\mathbb{M}^{2,1}$. Then its Gauss map is a harmonic map from $\mathbb{C}$ to $\mathbb{H}^2$ with the Hopf differential is the given $q$ and the $\partial$-energy density is the given $e^{2w}|dz|^2$, hence $\partial u\neq 0$.

By \cite{WAN}, since the space-like immersion $f$ is complete with respect to the induced Riemannian metric $e^{2w}|dz|^2$, then $f$ is an entire graph, meaning that the natural projection $\Pi:f(\Omega)\rightarrow \mathbb{R}^2$ is onto. Because the immersion $f$ is entire, by Theorem 4.8 in \cite{ChoiTreibergs}, the Gauss map of the immersion is a harmonic diffeomorphism onto its image. Therefore there exists an orientation-preserving harmonic embedding with the prescribed Hopf differential $q$. 

Next, we show the uniqueness. From Theorem \ref{Unique}, the unique solution $w$ to Equation (\ref{harmonic}) satisfies $e^{-4w}|\phi|^2<1$. So the Jacobian $J(u)=e^{2w}-|\phi|^2e^{-2w}>0$. Hence any harmonic map $u:\mathbb{C}\rightarrow \mathbb{H}^2$ satisfying $\partial u\neq 0$  with a polynomial Hopf differential $q$ is an orientation-preserving local diffeomorphism.

Suppose there are two harmonic maps $u_1,u_2:\mathbb{C}\rightarrow \mathbb{H}^2$ satisfying $\partial u_i\neq 0$ inducing the same polynomial Hopf differential $q$. Then they are both orientation-preserving local diffeomorphisms. Therefore, there exist two domains $\Omega_1,\Omega_2\in \mathbb{H}^2$ such that $u_2^{-1}:\Omega_2\rightarrow \mathbb{C}$ is well-defined and $u_1\circ u_2^{-1}:\Omega_2\rightarrow \Omega_1$ is a diffeomorphism. By Theorem \ref{Unique}, the $\partial$-energy densities for $u_1,u_2$ are the same. Then the pullback metrics of both maps $u_1,u_2$ is the same, written as $\phi dz^2+(e^{2w}+|\phi|^2e^{-2w})|dz|^2+\bar\phi d\bar z^2$. Hence the map $u_1\circ u_2^{-1}:\Omega_2\rightarrow \Omega_1$ is an isometry. Since the hyperbolic metric on $\mathbb{H}^2$ is analytic, $u_1\circ u_2^{-1}$ is a restriction of an isometry $\tau$ of $\mathbb{H}^2$. Therefore $u_1$ and $\tau\circ u_2$ coincide on a domain $u_2(\Omega)$. Then by the unique continuation theorem of harmonic maps shown in \cite{Sampson}, $u_1$ and $\tau\circ u_2$ are identical. The uniqueness follows.\end{proof}

\begin{rem}
When $q$ is a polynomial of degree $k$, the image of the unique harmonic map is an ideal (k+2)-polygon in $\mathbb{H}^2$, as shown in \cite{HAN,HTTW}.
\end{rem}

Next, we study the case when $q$ is not a polynomial.

\begin{thm}
Given a non-polynomial entire function $q$ with finite zeros, there are at least two harmonic maps $u_1,u_2:\mathbb{C}\rightarrow \mathbb{H}^2$ satisfying $\partial u_i\neq 0$ with the given $qdz^2$ as their Hopf differential not related by an isometry of $\mathbb{H}^2$, where $u_1$ is an orientation-preserving harmonic embedding.
\end{thm}
\begin{proof} Applying Theorem \ref{NonUnique}, we have two distinct solutions $w_1,w_2$ to Equation (\ref{harmonic}) for the same holomorphic quadratic differential $qdz^2$. Then by the argument in the proof of Theorem \ref{UniqueHarmonic}, both pairs $(w_1,q)$ and $(w_2,q)$ develop parabolic constant mean curvature immersions from $\mathbb{C}$ into $\mathbb{M}^{2,1}$ whose Gauss maps are harmonic maps $u_1, u_2$ from $\mathbb{C}$ to $\mathbb{H}^2$ giving back these two given pairs $(w_1,q),(w_2,q)$ respectively. As in the proof in Theorem \ref{UniqueHarmonic}, the harmonic map $u_1$ is an orientation-preserving harmonic embedding. 
\end{proof}

\section{Hyperbolic affine spherical immersion}
For a non-compact simply connected $2$-manifold $M$, consider a locally strictly convex immersed hypersurface $f:M\rightarrow \mathbb{R}^3$. Affine differential geometry associates to such a locally convex hypersurface a transversal vector field $\xi$, called the affine normal. Being an affine spherical immersion means the affine normal meets at a point (the center). By applying a translation, we can move the center of the affine sphere to the origin and write $\xi(p)=-Hp$, for all $p\in f(M)\subset \mathbb{R}^3$ for some constant $H\in \mathbb{R}$, the affine curvature. In the case when $H$ is negative, call the affine spherical immersion hyperbolic. After renormalization, we obtain a hyperbolic affine spherical immersion with center $0$ and of affine curvature $-1$.

 Decomposing the standard connection $D$ of $\mathbb{R}^3$ into tangent direction of $f(M)$ and affine normal components:
\[D_XY=\nabla_XY+h(X,Y)\xi, \quad \forall X,Y\in T_{f(p)}f(M).\] 
The second fundamental form $h$ of the image $f(M)$ relative to the affine normal $\xi$ can define a Riemannian metric $h$ on $M$, the Blaschke metric. This induces a complex structure on $M$. Also, the decomposition defines an induced connection $\nabla$ on $TM$. Let $\nabla^h$ be the Levi-Civita connection of the Blaschke metric $h$ and the Pick form $A(X,Y,Z)=h((\nabla-\nabla^h)_XY,Z)$ is a $3$-tensor, which uniquely determines a cubic differential $U=U(z)dz^3$ such that $\Re U=A$, the Pick differential. We focus on the case where $(M,h)$ is conformal to $(\mathbb{C},|dz|^2)$. In this case, we can reparametrize the hyperbolic affine spherical immersion $f:\mathbb{C}\rightarrow \mathbb{R}^3$ with the Blaschke metric is $e^w|dz|^2$.

Write the Blaschke metric $h=e^w|dz|^2$ and $U=U(z)dz^3$. Denoting $f_z, f_{\bar z}$ for $f_*(\frac{\partial}{\partial z}), f_*(\frac{\partial}{\partial \bar z})$ respectively.  Following Wang \cite{Wang} and Simon-Wang \cite{SimonWang}, the immersion $f$ being an affine spherical immersion 
implies that the frame field $(f, f_z,f_{\overline{z}})$ satisfying a linear first-order system of PDEs
\begin{eqnarray*}\label{system}
\frac{\partial}{\partial z}\begin{pmatrix}f\\f_z\\f_{\bar z}\end{pmatrix}=\begin{pmatrix}0&1&0\\0&w_z&Ue^{-w}\\ \frac{1}{2}e^w&0&0\end{pmatrix}\begin{pmatrix}f\\f_z\\f_{\bar z}\end{pmatrix},\\
\frac{\partial}{\partial \bar z}\begin{pmatrix}f\\f_z\\f_{\bar z}\end{pmatrix}=\begin{pmatrix}0&0&1\\\frac{1}{2}e^w&0&0\\ 0&\bar Ue^{-w}&w_{\bar z}\end{pmatrix}\begin{pmatrix}f\\f_z\\f_{\bar z}.\end{pmatrix}\end{eqnarray*}

Hence, the integrability for the above system (i.e. $f_{z\bar z}=f_{\bar z z}$)
 is equivalent to $U$ is holomorphic and 
\begin{equation}\label{WangEquation}
\triangle w=2e^w-4|U|^2e^{-2w}.\end{equation}
The curvature of the Blaschke metric $h=e^w|dz|^2$ is \[k_h=-\frac{1}{2}\triangle_h w=-1+2|U|^2e^{-3w}.\]

Conversely, given a function $w:\mathbb{C} \rightarrow \mathbb{R}$ and a holomorphic cubic differential $U=U(z)dz^3$ satisfying Equation (\ref{main}), we can integrate the linear first-order system (\ref{system}) and develop an affine spherical immersion $f:\mathbb{C}\rightarrow \mathbb{R}^3$  since the domain $\mathbb{C}$ is simply connected. And $f$ has Pick differential $U=U(z)dz^3$ and Blaschke metric $e^w|dz|^2$. 

In conclusion, any pair $(w,U)$ satisfying $U$ is holomorphic and Equation (\ref{WangEquation})
 is in one-to-one correspondence with an affine spherical immersion $f:\mathbb{C}\rightarrow \mathbb{R}^3$ with the Blaschke metric $e^w|dz|^2$ and the Pick differential $Udz^3$ up to a projective transformation of $\mathbb{R}^3$.

\begin{prop}\label{CompleteEmbedding}(\cite{Li90,Li92}) If the Blaschke metric of an affine spherical immersion $f:\mathbb{C}\rightarrow \mathbb{R}^3$ is complete, then $f$ is a proper embedding, and its image is asymptotic to the boundary of an open convex set.
\end{prop}

We then drop the condition on the completeness of the Blaschke metric.

\begin{thm}\label{embedding} 
 If $U$ is a polynomial function, there is a unique affine spherical immersion $f:\mathbb{C}\rightarrow \mathbb{R}^3$ with its Pick differential as $Udz^3$ up to a projective transformation of $\mathbb{R}^3$. Moreover, it is complete, properly embedded and its Blaschke metric is negatively curved.
\end{thm}

\begin{proof}
Applying Theorem \ref{Unique}, the polynomial $U$ uniquely determines the pair $(w,U)$ satisfying Equation (\ref{WangEquation}) Then by the discussion before Prop \ref{CompleteEmbedding}, there is a unique affine spherical immersion $f:\mathbb{C}\rightarrow \mathbb{R}^3$ with its Pick differential $Udz^3$ up to a projective transformation of $\mathbb{R}^3$. Moreover since the Blaschke metric $e^w|dz|^2$ is complete, by Proposition \ref{CompleteEmbedding}, $f$ is a proper embedding. The curvature of Blaschke metric is $k_h=-1+2|U|^2e^{-3w}$, by $2|U|^2e^{-3w}<1$, we obtain that the curvature of Blaschke metric $k_h<0$.
\end{proof}

Theorem \ref{embedding} gives a negative answer to the following question.
\begin{question}(\cite{DumasWolf}) Does there exist an affine spherical immersion $f:\mathbb{C}\rightarrow \mathbb{R}^3$ with polynomial Pick differential whose Blaschke metric has positive curvature at some point?
\end{question}
\begin{rem}
When $U$ is a polynomial of degree $k$, the image of the properly embedded affine sphere is asymptotic to a cone generated by a $(k+3)$-polygon in $\mathbb{R}^2$, as shown in \cite{DumasWolf}.
\end{rem}

Next, we study the case when $U$ is not a polynomial.
\begin{thm}
Given a non-polynomial entire function $U$ with finite zeros,  then there are at least two affine spherical immersions $f_1,f_2:\mathbb{C}\rightarrow \mathbb{R}^3$ with the given $Udz^3$ as their Pick differential, where \\
(1) $f_1$ is properly embedded. The Blaschke metric is complete and negatively curved; while\\
(2) $f_2$ is not properly embedded. The Blaschke metric is incomplete and nonpositively curved.
\end{thm}
\begin{proof}
Applying Theorem \ref{NonUnique}, we obtain two solutions $(w_1,w_2)$ to Equation where $e^{w_1}|dz|^2$ is complete while $e^{w_2}|dz|^2$ is incomplete. Then the two pairs $(w_1,Udz^3)$ and $(w_2,Udz^3)$ develop two affine spherical immersions $f_1,f_2: \mathbb{C}\rightarrow \mathbb{R}^3$ respectively. Since $e^{w_1}|dz|^2$ is complete, by Proposition \ref{CompleteEmbedding}, $f_1$ is a proper embedding.  By Theorem \ref{NonUnique}, $w_1$ satisfies $2|U|^2e^{-3w_1}<1$, implying the curvature $k_1=-1+2|U|^2e^{-3w_1}<0$; $w_2$ satisfies $2|U|^2e^{-3w_2}\leq1$, implying the curvature $k_2=-1+2|U|^2e^{-3w_2}\leq 0$. 
\end{proof}

\bibliographystyle{amsalpha}
\bibliography{bib}
\end{document}